\documentclass[a4paper,11pt]{article}
\usepackage[utf8]{inputenc}

\usepackage{boxedminipage}
\usepackage{amsfonts}
\usepackage{amsmath} 
\usepackage{amssymb}
\usepackage{graphicx}
\usepackage{amsthm}
\usepackage{t1enc}
\usepackage{subfig}
\usepackage{cancel}
 \usepackage{textcmds}
 \usepackage{mathabx}
 \usepackage{url}

\newtheorem{theorem}{Theorem}[section]
\newtheorem{lemma}[theorem]{Lemma}

\newtheorem{definition}[theorem]{Definition}

\newtheorem*{theorem*}{Theorem}

\newcommand{\forceP}{\mathbb{P}}
\newcommand{\forceQ}{\mathbb{Q}}
\newcommand{\forceR}{\mathbb{R}}

\newcommand{\ZFC}{\mathsf{ZFC}}

\newcommand{\ZFP}{\mathsf{ZF}^-}

\newcommand{\MRP}{\mathsf{MRP}}

\newcommand{\CH}{\mathsf{CH}}

\newcommand{\NS}{\text{NS}_{\omega_1}}

\newcommand{\AD}{\mathsf{AD}}
\newcommand{\mouseM}{\mathcal{M}}
\newcommand{\mouseN}{\mathcal{N}}

\def\undertilde#1{\mathord{\vtop{\ialign{##\crcr
$\hfil\displaystyle{#1}\hfil$\crcr\noalign{\kern1.5pt\nointerlineskip}
$\hfil\tilde{}\hfil$\crcr\noalign{\kern1.5pt}}}}}

\title{$\text{NS}_{\omega_1}$ saturated, $\Delta_1 ( \{ \omega_1 \} )$-definable and a $\Delta^1_4$-definable well-order of the reals}
\author{ Stefan Hoffelner$^{1}$\footnote{This research was funded in whole by the Austrian Science Fund (FWF) Grant-DOI 10.55776/P37228. }  }
\date{
    $^1$TU Wien \\
    \today }
\begin{document}

\maketitle

\begin{abstract}
Assuming $M_1$, the canonical inner model with one Woodin cardinal exists, we construct a model in which the nonstationary ideal on $\omega_1$ is $\aleph_2$-saturated, $\Delta_1$-definable with $\omega_1$ as the only parameter and there is a
$\Sigma^1_{4}$-definable well-order of the reals. This implies that contrary to the assumption that $\NS$ is $\aleph_1$-dense, the assumption of $\NS$ being saturated and $\Delta_1$-definable does not imply any nice structural properties for the projective subsets of the reals.
\end{abstract}

\section{Introduction}

This work deals with definability questions in the presence of the nonstationary ideal on $\omega_1$ ($\NS$) being $\aleph_2$-saturated. Recall that $\NS$ is ($\aleph_2$-)saturated if every antichain in the Boolean algebra $P(\omega_1) \slash \NS$ has length $\le \aleph_1$. In other words, every sequence of stationary sets on $\omega_1$, $\vec{S} = (S_{\alpha} \mid \alpha < \kappa)$ with the additional property that for every $\alpha \ne \beta$, $S_{\alpha} \cap S_{\beta}$ is nonstationary, has length at most $\omega_1$. It is safe to say that the question of the possible consistency of $\NS$ being saturated turned out to be one of the most fruitful in the history of set theory, and greatly contributed to the formulation and development of many, now central notions of modern set theory (among them the forcing axiom Martin's Maximum and the $\forceP_{\text{max}}$-technique).
These developments showed that there are deep connections between strong structural properties of the nonstationary ideal on $\omega_1$ and the behaviour of the real numbers which are mostly due to H. W. Woodin, with ``$\NS$ is saturated and there is a measurable cardinal$"$ implies that $\CH$ fails, being arguably the most prominent example of these results.

The main motivation for this article comes from his following result (see \cite{W}, theorem 6.81) . Recall first that the nonstationary ideal is $\omega_1$-dense if the Boolean algebra $P(\omega_1) \slash \NS$ has a dense subset of size $\aleph_1$. 
\begin{theorem*}[Woodin]
Assume that $\NS$ is $\omega_1$-dense, then $\AD$ holds in $L(\mathbb{R})$.
On the other hand, if $\AD$ is true in $L(\mathbb{R})$ then forcing with $\forceQ^{*}_{\text{max}}$  over $L(\mathbb{R})$ will produce a model of $\ZFC$ where $\NS$ is $\omega_1$-dense.
\end{theorem*}
In particular, if $\NS$ is $\omega_1$-dense, then there is no projectively definable well-order of the reals. 

It is natural to ask whether the assumption of $\NS$ being $\omega_1$-dense can be weakened to still obtain some, possibly limited, nice structural properties of the projective sets of reals.
Note that $\NS$ being $\omega_1$-dense has two main implications for the nonstationary ideal, namely it implies that $\NS$ is saturated, and also  $\NS$ becomes $\Delta_1$ definable over $H(\omega_2)$, using the dense family as a parameter. Goal of this article is the proof of the following theorem, which shows that $\Delta_1$-definability of $\NS$ together with $\NS$ being saturated alone do not suffice to impose nice structural properties on the projective sets of reals.

\begin{theorem*}
Work over $M_1$, the canonical inner model with one Woodin cardinal. Then there is a generic extension of $M_1$ in which 
\begin{enumerate}
\item $\NS$ is saturated,
\item $\NS$ is $\Delta_1$-definable over $H(\omega_2)$ with $\omega_1$ as the only parameter,
\item there is a $\Delta^1_4$-definable well-order of the reals.
\end{enumerate}
\end{theorem*}

This work builds on the model constructed in \cite{NS Delta_1}, where the first two items of the theorem hold. The question of one can have a model with an additional $\Delta^1_4$-definable well-order of its reals remained an open question there.  A model where the first and the third item hold is produced in \cite{FH}, however, due to technical reasons created by the ``localization forcings$"$ used there, the two methods from \cite{NS Delta_1} and \cite{FH} can not be combined.   We shall refine the construction from \cite{NS Delta_1}, using ideas from \cite{Pi Uniformization}, in order to produce the desired definable well-order, along with a $\Sigma_1 (\{ \omega_1 \} )$-predicate of stationarity on $\omega_1$.

In a wider context, this article adds to the growing literature which investigates the definability of $\NS$ in various contexts (see \cite{LueckeSchindlerSchlicht}, \cite{FL}, \cite{HLSW1}, \cite{HLSW2} or \cite{Lu}),
which started with A. H. Mekler and S. Shelah and T. Hyttinen and M. Rautila (see \cite{MekShe} and \cite{HytRau}). The many techniques and connections with other parts of set theory testify that the investigation of the definability of $\NS$ is a very interesting one.

\section{Towards a suitable ground model}

Attempting to use traditional coding arguments—which are typically designed for G\"odel's constructible universe $L$—within the framework of canonical inner models possessing a finite number of Woodin cardinals ($M_n$) presents considerable technical challenges. These difficulties are particularly acute when coding forcings are iterated in a manner that causes the largest Woodin cardinal in $M_n$ to become a successor cardinal in the resulting generic extension.

To overcome these issues, one strategy involves generically modifying the ground models $M_n$. This creates new universes over which coding arguments can be effectively applied. This is the precise approach taken in the work \cite{NS Delta_1}. Since we will use the exact same generic extension of $M_1$ (denoted as $W_0$) as our foundation for a coding argument, we will now detail its construction.

The universe $W_0$ is constructed by forcing over $M_1$ using three distinct types of forcings, which will be introduced shortly. Two of these are coding forcings, employed to endow $W_0$ with specific structures. These structures are essential for a subsequent, second iteration of coding forcings that will be applied over $W_0$. This second machinery is the forcing designed to produce both a $\Delta^1_4$-definable well-order and a $\Delta_1 ( \{ \omega_1 \})$-definition of stationarity.

The third type of forcing aims to make the nonstationary ideal saturated. This is achieved through iterations of sealing antichains forcings, as in the Martin's Maximum paper by Foreman, Magidor and Shelah (see \cite{FMS}).

Let's now focus on defining the coding forcings used in the formation of $W_0$.
We utilize the coding technique developed by A. Caicedo and B. Velickovic (see \cite{CV}). Sets coded using their method serve as a suitable ground-model $H(\omega_2)$. Over this $H(\omega_2)$, we then apply further coding using ccc  forcings. In essence, as mentioned, we are coding on top of already coded sets. A convenient property of the Caicedo-Velickovic coding is that ccc forcings cannot introduce new codes. This has the beneficial consequence that the ground-model $H(\omega_2)$ can always be defined in ccc generic extensions. Specifically, this allows us to identify objects from the ground-model $H(\omega_2)$ (in our case, Suslin trees) and modify them generically—for instance, by adding cofinal branches or specializing functions using ccc forcings. This ultimately yields a new coding mechanism whose codes can be easily decoded.

Although many proofs will be omitted in the following discussion (these can be found in \cite{NS Delta_1}), it is necessary to describe the Caicedo-Velickovic coding in detail, as its features are crucial for understanding this article.

\subsection{Coding Reals by Triples of Ordinals}

We will now present the coding method conceived by A. Caicedo and B. Velickovic (see \cite{CV}) that will be employed in our argument.

\begin{definition}
A $\vec{C}$-sequence, also known as a ladder system, is a sequence $(C_{\alpha} \, : \, \alpha \in \omega_1, \alpha \text{ is a limit ordinal})$, such that for every $\alpha$, $C_{\alpha}$ is a cofinal subset of $\alpha$ and has an ordertype of $\omega$.
\end{definition}

Given three subsets of natural numbers, $x, y, z \subset \omega$, we can define an oscillation function. First, the set $x$ is transformed into an equivalence relation $\sim_x$ defined on $\omega - x$. For natural numbers $n, m$ in the complement of $x$ where $n \le m$, we define $n \sim_x m$ if and only if the interval $[n,m]$ has no elements in common with $x$ (i.e., $[n,m] \cap x = \emptyset$).
This allows us to define the following:

\begin{definition}
For a triple of subsets of natural numbers $(x,y,z)$, let $(I_n \, :\, n \in k \le \omega)$ be the sequence of equivalence classes of $\sim_x$ that have a non-empty intersection with both $y$ and $z$. The oscillation map $o(x,y,z): k \rightarrow 2$ is then defined as the function satisfying:

$$o(x,y,z)(n) = \begin{cases} 0 & \text{if } \min(I_n \cap y) \le \min(I_n \cap z) \\ 1 & \text{otherwise} \end{cases}$$
\end{definition}

Next, we will define how appropriate countable subsets of ordinals can be used to code real numbers. For the remainder of this section, we fix a ladder system $\vec{C}$.
Suppose $\omega_1 < \beta < \gamma < \delta$ are fixed limit ordinals, each with uncountable cofinality. Let $N \subset M$ be countable subsets of $\delta$.
Furthermore, assume that $\{ \omega_1, \beta, \gamma\} \subset N$, and for each $\eta \in \{ \omega_1, \beta, \gamma\}$, $M \cap \eta$ is a limit ordinal and $N \cap \eta < M \cap \eta$.

We can use the pair $(N,M)$ to code a finite binary string. Let $\bar{M}$ be the transitive collapse of $M$, and let $\pi : M \rightarrow \bar{M}$ be the collapsing map. Define $\alpha_M := \pi(\omega_1)$, $\beta_M := \pi(\beta)$, $\gamma_M := \pi(\gamma)$, and $\delta_M := \bar{M}$. These are all countable limit ordinals.
Also, set $\alpha_N := \sup(\pi``(\omega_1 \cap N))$. The height $n(N,M)$ of $\alpha_N$ in $\alpha_M$ is the natural number defined by:

$$n(N,M) := \text{card}(\alpha_N \cap C_{\alpha_M})$$
where $C_{\alpha_M}$ is an element of our pre-fixed ladder system. Since $n(N,M)$ will appear frequently, we will abbreviate it as $n$. Note that because the ordertype of each $C_{\alpha}$ is $\omega$, and $N \cap \omega_1$ is bounded below $M \cap \omega_1$, $n(N,M)$ is indeed a natural number.

Now, we can assign a triple $(x,y,z)$ of finite subsets of natural numbers to the pair $(N,M)$ as follows:
$$x := \{ \text{card}(\pi(\xi) \cap C_{\beta_M}) \, : \, \xi \in \beta \cap N \}$$
Note that $x$ is finite because $\beta \cap N$ is bounded in the set $C_{\beta_M}$ (which is cofinal in $\beta_M$ and has ordertype $\omega$). Similarly, we define:
$$y := \{ \text{card}(\pi(\xi) \cap C_{\gamma_M}) \, : \, \xi \in \gamma \cap N \}$$
and
$$z := \{ \text{card}(\pi(\xi) \cap C_{\delta_M}) \, : \, \xi \in \delta \cap N \}$$
Again, it is straightforward to see that these are finite subsets of natural numbers.

We can then consider the oscillation $o(x \setminus n, y \setminus n, z \setminus n)$ (recalling that $n = n(N,M)$). If the domain of the oscillation function at these points is greater than or equal to $n$, we write:
$$s_{\beta, \gamma, \delta} (N,M) := \begin{cases} o(x \setminus n, y \setminus n, z \setminus n)\upharpoonright n & \text{if defined} \\ \ast & \text{otherwise} \end{cases}$$
Here, $\ast$ simply represents an error symbol. Similarly, if $l > n$, we let $s_{\beta, \gamma, \delta} (N,M) \upharpoonright l = \ast$.

Finally, we can define what it means for a triple of ordinals $(\beta, \gamma, \delta)$ to code a real number $r$.

\begin{definition}
For a triple of limit ordinals $\omega_1 < \beta < \gamma < \delta$, each with uncountable cofinality, we say that it codes a real $r \in 2^{\omega}$ if there exists a continuous, increasing sequence $(N_{\xi} \, : \, \xi < \omega_1)$ of countable sets of ordinals whose union is $\delta$. This sequence must also satisfy the condition that there is a club $C \subset \omega_1$ such that whenever $\xi \in C$ is a limit ordinal, there exists a $\nu < \xi$ for which:
$$r = \bigcup_{\nu < \eta < \xi} s_{\beta, \gamma, \delta} (N_{\eta}, N_{\xi})$$
We call the sequence $(N_{\xi} \, : \, \xi < \omega_1)$ a reflecting sequence.
\end{definition}

Witnesses for this coding can be added using a proper forcing. Conversely, for fixed triples of ordinals, there is a degree of control over the behavior of continuous, increasing sequences on them:

\begin{theorem}[Caicedo-Velickovic]
\begin{itemize}
\item[$(\dagger)$] Given ordinals $\omega_1 < \beta < \gamma < \delta < \omega_2$, each with cofinality $\omega_1$, there exists a proper notion of forcing $\forceP_{\beta \gamma \delta}$ such that after forcing with it, the following holds: There is an increasing, continuous sequence $(N_{\xi} \, : \, \xi < \omega_1)$ where $N_{\xi} \in [\delta]^{\omega}$ and their union is $\delta$. This sequence is such that for every limit ordinal $\xi < \omega_1$ and every $n \in \omega$, there exist $\nu < \xi$ and $s_{\xi}^n \in 2^n$ satisfying $s_{\beta \gamma \delta}(N_{\eta}, N_{\xi}) \upharpoonright n = s_{\xi}^{n}$ for every $\eta$ in the interval $(\nu, \xi)$. In this case, we say the triple $(\beta, \gamma, \delta)$ is stabilized.

\item[$(\ddagger)$] Furthermore, if we fix a real $r$, there is a proper notion of forcing $\forceP_r$. This forcing will produce, for a triple of ordinals $(\beta_r, \gamma_r, \delta_r)$ of size and cofinality $\omega_1$, a reflecting sequence $(P_{\xi} \, : \, \xi < \omega_1)$, where $P_{\xi} \in [\delta_r]^{\omega}$, $\bigcup P_{\xi} = \delta_r$. Additionally, there is a club $C \subset \omega_1$ such that for every limit ordinal $\xi \in C$, there is a $\nu < \xi$ for which:
    $$\bigcup_{\nu < \eta < \xi} s_{\beta_r \gamma_r \delta_r} (P_{\eta}, P_{\xi}) = r.$$

\end{itemize}
  
    \end{theorem}

Both partial orders $\forceP_{\beta \gamma \delta}$ and $\forceP_{r}$, which force $(\dagger)$ and $(\ddagger)$ respectively, are specific instances of a general class of forcing notions initially investigated by J. Moore in his work on the Set Mapping Reflection Principle ($\MRP$) (see \cite{Moore}). It is crucial for our purposes that these forcings and their iterations preserve Suslin trees.

\begin{theorem}
\begin{enumerate}
\item Assume that $\omega_1 < \beta < \gamma < \delta < \omega_2$ and each ordinal has uncountable cofinality. Then the forcing $\forceP_{\beta \gamma \delta}$ preserves Suslin trees from the ground model.
\item Let $r$ be a real. Then the forcing $\forceP_r$ preserves Suslin trees from the ground model.
\end{enumerate}

\end{theorem}
The proof of this theorem can be found in \cite{NS Delta_1}.

\subsection{Almost Disjoint Coding}

The second coding method we will use is the almost disjoint coding forcing, developed by R. Jensen and R. Solovay \cite{JensenSolovay}. We will identify subsets of $\omega$ with their characteristic functions and use the term "reals" for elements of $2^{\omega}$ and subsets of $\omega$ interchangeably.

Let $F=\{f_{\alpha} \, : \, \alpha < 2^{\aleph_0} \}$ be a family of almost disjoint subsets of $\omega$; that is, a family such that if $r, s \in F$, then their intersection $r \cap s$ is finite. Let $X \subset \kappa$ (where $\kappa \le 2^{\aleph_0}$) be a set of ordinals. There exists a ccc forcing, known as the almost disjoint coding $\mathbb{A}_F(X)$, which adds a new real $x$. This real $x$ codes the set $X$ relative to the family $F$ in the following manner:
$$\alpha \in X \text{ if and only if } x \cap f_{\alpha} \text{ is finite.}$$

\begin{definition}
The almost disjoint coding $\mathbb{A}_F(X)$ relative to an almost disjoint family $F$ consists of conditions $(r, R)$, where $r \in \omega^{<\omega}$ (a finite sequence of natural numbers) and $R \in F^{<\omega}$ (a finite sequence of elements from $F$). A condition $(s,S)$ is stronger than $(r,R)$ (denoted $(s,S) < (r,R)$) if and only if:
\begin{enumerate}
\item $r$ is an initial segment of $s$ ($r \subset s$), and $R$ is an initial segment of $S$ ($R \subset S$).
\item  If $\alpha \in X$, then $r \cap f_{\alpha} = s \cap f_{\alpha}$.
\end{enumerate}

\end{definition}

Another variant, possibly due to L. Harrington (see \cite{Harrington}), codes sets of reals relative to a new real. For the following, we fix some definable bijection between finite sequences of integers and $\omega$. For $b \in \omega^{\omega}$, let $\bar{b}(n)$ denote the natural number that codes the finite sequence $b \upharpoonright n$ (the initial segment of $b$ of length $n$). A subset $b \subset \omega$ (viewed as a real) gives rise to a new set $S(b) \subset \omega$ if we consider the set of codes of its initial segments: $\{ \bar{b}(n) \, : \, n \in \omega \}$.

\begin{definition}
Suppose $A \subset [\omega]^{\omega}$ (i.e., $A$ is a set of infinite subsets of $\omega$). The almost disjoint coding forcing for $A$, denoted $\mathbb{A}(A)$, is defined as follows. Conditions are pairs $(s(0), s(1))$ such that $s(0)$ is a finite set of natural numbers and $s(1)$ is a finite subset of the fixed set of reals $A$. For two conditions $p, q \in \mathbb{A}(A)$, we say $q < p$ ( $q$ is stronger than $p$) if and only if:
\begin{itemize}
\item $p(0) \subset q(0)$ and $p(1) \subset q(1)$.
\item $\forall a \in p(1) \, (S(a) \cap q(0) \subset p(0))$.
\end{itemize}
 
\end{definition}

Note that $\mathbb{A}(A)$ has the Knaster property, which implies that products of $\mathbb{A}(A)$ satisfy the ccc. Given a set of reals $A$ in the ground model $V$, forcing with $\mathbb{A}(A)$ has the following effect: The first coordinates of conditions in the generic filter $G$ will combine to form a real $a$. This real $a$ codes the set $A$ with the help of a predicate for $\omega^{\omega} \cap V$. In the generic extension $V[G]$, membership in $A$ is characterized as:
$$x \in A \leftrightarrow x \in V \land S(x) \cap a \text{ is finite.}$$

This characterization will play a crucial role later. As previously mentioned, our goal is to work towards a universe whose $H(\omega_2)$ is definable in arbitrary generic extensions obtained using ccc forcings. Consequently, we will have definable access to $H(\omega_2)^V$ (the $H(\omega_2)$ of the ground model $V$). We then use the just-defined forcing to encode information into a single real, ensuring that this information can be correctly decoded in all subsequent ccc extensions of the universe.

\subsection{$\NS$ saturated}

As our proof depends on Shelah's argument to force $\NS$ saturated from a Woodin cardinal we introduce very briefly
some of the main ideas. 
The crucial forcing notion which can be used to bound the length 
of antichains in $P(\omega_1) \slash \NS$ is the sealing forcing.

\begin{definition}
 Let $\vec{S}=(S_i \, : \, i < \kappa)$ be a maximal antichain in
 $P(\omega_1) \slash \NS$. Then the sealing forcing for $\vec{S}$, $\mathbb{S} (\vec{S})$
 is defined as follows. Conditions
 are pairs $(p,c)$ such that $p: \alpha +1 \rightarrow \vec{S}$ and 
 $c: \alpha+1 \rightarrow \omega_1$, where the image of $c$ should be closed, and
 $\alpha < \omega_1$. We additionally demand that 
 $\forall \xi < \omega_1 \, c(\xi) \in \bigcup_{i \in \xi} p(i)$, and conditions are ordered by
 end-extension.
\end{definition}
Thus, given a maximal antichain $\vec{S}$, $\mathbb{S} (\vec{S} )$ will collapse its length down to $\omega_1$ while simultaneously shoot a club through the diagonal union of $\vec{S}$. The latter has the desired effect that $\vec{S}$ remains a maximal antichain in all stationary set preserving outer models, which is wrong if we would just collapse the length of $\vec{S}$ down to $\omega_1$. It is wellknown that $\mathbb{S} (\vec{S} )$ is $\omega$-distributive and 
stationary sets preserving if and only if $\vec{S}$ is maximal.

\begin{theorem}(Shelah)
 Let $V$ be a universe with a Woodin cardinal $\delta$. Then there is a $\delta$-sized 
 forcing notion $\forceP$, which is an RCS-iteration of length $\delta$  of sealing forcings, such that in $V^{\forceP}$, $\NS$ is saturated
 and $\omega_2 = \delta$.
\end{theorem}

Again it is crucial for our needs that the sealing forcings $\mathbb{S} (\vec{S} )$ do preserve Suslin trees.
\begin{theorem}
Let $\vec{S}$ be a maximal antichain in $P(\omega_1) \slash \NS$ then the sealing forcing $\mathbb{S} (\vec{S} )$ preserves ground model Suslin trees.
\end{theorem}
A proof of the theorem can again be found in \cite{NS Delta_1}.

\subsection{Independent Suslin trees}
Suslin trees are used to construct the last of the three coding techniques we will use in
the proof. 
Recall that a set theoretic tree $(T, <)$ is a Suslin tree if it is a normal tree of height $\omega_1$
and no uncountable antichain. All the trees which appear in this paper will be normal, thus whenever we talk about trees it is implicitly assumed that these trees are normal.

Recall that for two trees $(T_0, <_{T_0})$ and $(T_1, <_{T_1})$ their tree-product $T_0 \times T_1$ is defined to be the tree which consists of nodes $\{ (t_0,t_1) \, : \, t_0 \in T_0 \land t_1 \in T_1 \land$ height$(t_0) =$ height$(t_1) \}$, ordered by $(t_0,t_1) <_{T_0 \times T_1} (s_0, s_1)$ if and only if $t_0 <_{T_0} s_0$ and $t_1 <_{T_1} s_1$. From now on whenever we talk about a product of trees, it is always the tree product which is meant.
For our purposes it is necessary to iteratively add 
sequences of blocks of Suslin trees $(\bar{T}_{\alpha} \, : \, \alpha < \kappa)$
such that $\bar{T}$ is itself an $\omega$-length sequence of Suslin trees whose finite subproducts are Suslin again.
One can construct such sequences using Jech's forcing which adds a
Suslin tree with countable conditions.

\begin{definition}
 Let $\forceP_J$ be the forcing whose conditions are
 countable, normal trees ordered by end-extension, i.e. $T_1 < T_2$ if and only
 if $\exists \alpha < \text{height}(T_1) \, T_2= \{ t \upharpoonright \alpha \, : \, t \in T_1 \}$
\end{definition}
It is wellknown that $\forceP_J$ is $\sigma$-closed and
adds a Suslin tree. In fact more is true, the generically added tree $T$ has 
the additional property that for any Suslin tree $S$ in the ground model
$S \times T$ will be a Suslin tree in $V[G]$.
\begin{lemma}
 Let $V$ be a universe and let $S \in V$ be a Suslin tree. If $\forceP_J$ is 
 Jech's forcing for adding a Suslin tree and if $T$ is the generic tree
 then $$V[T] \models T \times S \text{ is Suslin.}$$
\end{lemma}

\begin{proof}
Let $\dot{T}$ be the $\forceP_J$-name for the generic Suslin tree. We claim that $\forceP_J \ast \dot{T}$ has a dense subset which is $\sigma$-closed. As $\sigma$-closed forcings will always preserve ground model Suslin trees, this is sufficient. To see why the claim is true consider the following set:
$$\{ (p, \check{q}) \, : \, p \in \forceP_J \land height(p)= \alpha+1  \land  \check{q} \text{ is a node of $p$ of level } \alpha \}.$$
It is easy to check that this set is dense and $\sigma$-closed in $\forceP_J \ast \dot{T}$.

\end{proof}

A similar observation shows that a we can add an $\omega$-sequence of
such Suslin trees with a fully supported iteration. Even longer sequences
of such trees are possible if we lengthen the iteration but for our needs $\omega$-blocks are sufficient.

\begin{lemma}\label{ManySuslinTrees}
 Let $S$ be a Suslin tree in $V$ and let $\forceP$ be a fully supported
 iteration of length $\omega$ of forcings $\forceP_J$. Then in the generic extension
 $V[G]$ there is an $\omega$-sequence of Suslin trees $\vec{T}=(T_n \, : \, n \in \omega)$ such
that for any finite $e \subset \omega$
the tree $S \times \prod_{i \in e} T_i$ will be a Suslin tree in $V[\vec{T}]$.
\end{lemma}
\begin{proof}
Let $G$ be a generic filter for $\forceP$. First we observe that the fully supported product $\prod_{n \in \omega} \forceP_J$ followed by the fully supported forcing with the product of the generically added trees $T_n$ has a $\sigma$-closed dense subset which is defined coordinate-wise as above, thus ground model Suslin trees are preserved. If we pick a finite $e \subset \omega$, then for an arbitrary Suslin tree $S \in V$, $\prod_{i \in e} T_i \times S$ is a Suslin tree in the intermediate model generated by the trees $T_i$, $i \in e$ over $V$. This is preserved when passing to the generic extension $V[G]$.
\end{proof}

These sequences of Suslin trees will become important later in our proof, thus they will get a name.
\begin{definition}
 Let $\vec{T} = (T_{\alpha} \, : \, \alpha < \kappa)$ be a sequence of Suslin trees. We say that the sequence is an 
 independent family of Suslin trees if for every finite set $e= \{e_0, e_1,...,e_n\} \subset \kappa$ the product $T_{e_0} \times T_{e_1} \times \cdot \cdot \cdot \times T_{e_n}$ 
 is a Suslin tree again.
\end{definition}

\subsection{Some properties of the ground model $M_1$}
The ground model of our forcing construction is the canonical inner model with one Woodin cardinal $M_1$, thus we will introduce some of its properties which are crucial for our needs. We will not assume that the reader is familar with the basic notions of inner model theory, and instead state the important properties of $M_1$ which are proved in \cite{NS Delta_1}. Any reader will be able to follow the proofs of the article as long as she is willing to accept those properties as additional axioms.

 Recall that $M_1$ is a proper class premouse containing a Woodin cardinal (see \cite{Steel2}, pp. 81 for a definition of $M_1)$.
The reals of $M_1$ admit a $\Sigma^{1}_3$-definable wellorder (see \cite{Steel2}, Theorem 4.5), the definition of the wellorder makes crucial use of a weakened notion of iterability, the so-called $\Pi^1_2$-iterability which we shall not introduce and blackbox its main consequences. Again the reader can find the relevant proofs in \cite{NS Delta_1}.

The next lemma is folklore.
\begin{lemma}
Let $\mathcal{M}$ and $\mathcal{N}$ be $\omega$-sound premice which both project to $\omega$. Assume that $\mathcal{M}$ is an initial segment of $M_1$ and $\mathcal{N}$ is $\Pi^1_2$-iterable, and let $\Sigma$ denote the winning strategy for player II in $\mathcal{G}_{\omega} (\mathcal{M}, \omega_1+1)$. Then we can successfully compare $\mathcal{M}$ and $\mathcal{N}$ and consequently $\mathcal{M} \triangleleft \mathcal{N}$ or $\mathcal{N} \trianglelefteq \mathcal{M}$.
\end{lemma}

It is relatively straightforward to check that the set of reals which code $\Pi^1_2$-iterable, countable premice is itself a $\Pi^1_2$-definable set in the codes (see \cite{Steel2}, Lemma 1.7). Modulo the last lemma, this implies that there is a nice definition of a cofinal set of countable initial segments of $M_1$  in $\omega_1$-preserving
forcing extensions $M_1[G]$ of $M_1$, (in fact this definiton holds in all outer models of $M_1$ with the same $\omega_1$):

\begin{lemma}
Let $M_1[G]$ be an $\omega_1$-preserving forcing extension of $M_1$. Then in $M_1[G]$ there is
$\Pi^{1}_2$-definable set $\mathcal{I}$ of premice  which are of the form $\mathcal{J}^{M_1}_{\eta}$ for 
some $\eta< \omega_1$. $\mathcal{I}$ is defined as
$$\mathcal{I}:= \{ \mouseM \text{ ctbl premouse} \, : \,\mouseM \text{ is } \Pi^{1}_2\text{-iterable}, \, 
\omega\text{-sound} \text{ and projects to } \omega \},$$
and the set
$$\{ \eta < \omega_1 \, : \, \exists \mouseN \in \mathcal{I} (\mouseN = \mathcal{J}^{M_1}_{\eta})\}$$ is cofinal in $\omega_1$.
\end{lemma}
In particular $M_1| \omega_1$ is $\Sigma_1(\omega_1)$-definable in $\omega_1$-preserving generic extensions of $M_1$, as $x \in M_1 | \omega_1$ if and only if there is a transitive $U \models \ZFP$, $\omega_1 \subset U$, $\aleph_1^U=\aleph_1$ such that $U \models \exists \mathcal{M} \in \mathcal{I} \land x \in \mathcal{M}$, which suffices using Shoenfield absoluteness. A similar argument also shows that $\{ M_1 | \omega_1\}$ is $\Sigma_1(\omega_1)$ definable, as we can successfully compute it in transitive $\omega_1$-containing models, via the following $\Sigma_1(\omega_1)$-formula:
 \begin{align*}
(\ast) \quad X=M_1|\omega_1 \Leftrightarrow  \exists U (&U \text{ is a transitive model of } \ZFP \land \omega_1 \subset U
\land \\& U \models \forall \alpha < \omega_1 \exists r \in \mathcal{I} (\alpha \in (r \cap Ord)) \land \\& \, \,\qquad X \text{ is transitive and } X \cap Ord=\omega_1 \land 
\\& \qquad \qquad \forall x \in \mathcal{I} (x \subset X) \land \forall y \in X \exists x \in \mathcal{I} (y \in x))
\end{align*} 
Indeed, if the left hand side of $(\ast)$ is true, then any transitive $U$ which contains $M_1 | \omega_1$ as an element and which models $\ZFP$ will witness the truth of the right hand side, which is an immediate consequence of Shoenfield absoluteness.

If the right hand side is true, then, using the fact that $\Sigma^1_3$-statements are upwards absolute between $U$ and the real world, $U$ will contain an $\omega_1$-height, transitive structure $X$ which contains all countable initial segments of $M_1$, and such that every $y \in X$ is included in some element of $M_1| \omega_1$, in other words $X$ must equal $M_1 | \omega_1$.

We shall now apply the just obtained definability results to show that over $M_1$ there are $\Sigma_1(\omega_1)$-definitions of an $\omega$-sequence of independent Suslin trees and a ladder system on $\omega_1$.
 
The first thing to note is that $M_1 | \omega_1$ can define a $\diamondsuit$-sequence in the same way as $L_{\omega_1}$ can. Indeed, as $M_1$ has a $\Delta_3^1$-definable wellorder of the reals whose definition relativizes to $M_1 | \omega_1$ we can repeat Jensen's original proof in $M_1$ to construct a candidate for the $\diamondsuit$-sequence, via picking at every limit stage $\alpha< \omega_1$ the $<_{M_1}$-least pair $(a_{\alpha}, c_{\alpha}) \in P(\alpha) \times P(\alpha)$ which witnesses that the sequence we have created so far is not a $\diamondsuit$-sequence. The proof that this defines already a witness for $\diamondsuit$ is finished as usual with a condensation argument. Hence we shall show that if $\mathcal{J}^{M_1}_{\beta}$ is least such that $(a_{\alpha} \, : \, \alpha< \omega_1) $ and $(A,C) \in \mathcal{J}^{M_1}_{\beta}$, where $(A,C)$ is the $<_{M_1}$-least witness for $(a_{\alpha})_{\alpha < \omega_1}$ not being a $\diamondsuit$-sequence, then there is an countable $N \prec \mathcal{J}^{M_1}_{\beta}$ such that the transitive collapse $\bar{N}$ is an initial segment of $M_1$. 

To see that in fact every such $N$ collapses to an initial of $M_1$, recall
the condensation result as in \cite{Steel3}, Theorem 5.1, which we can state in our situation as follows:
\begin{theorem}
Let $\mathcal{M}$ be an initial segment of $M_1$. Suppose that $\pi: \bar{N} \rightarrow \mathcal{M}$ is the inverse of the transitive collapse and $crit(\pi)=\rho^{\bar{N}}_{\omega}$, then either
\begin{enumerate}
\item $\bar{N}$ is a proper initial segment of $\mathcal{M}$, or
\item there is an extender $E$ on the $\mathcal{M}$-sequence such that
$lh(E)=\rho^{\bar{N}}_{\omega}$, and $\bar{N}$ is a proper initial segment of $Ult_0(\mathcal{M},E)$.
\end{enumerate}
\end{theorem}
We shall argue, that in our situation, the second case is ruled out, hence every $N \prec \mathcal{J}^{M_1}_{\beta}$ collapses to an initial segment of $M_1$. Indeed, due to the $\omega$-soundness of $\mathcal{J}^{M_1}_{\beta}$, every $N \prec \mathcal{J}^{M_1}_{\beta}$ will satisfy that\[ \rho_{\omega}^N= \rho_{\omega}^{\mathcal{J}^{M_1}_{\beta}}=\omega_1^{\mathcal{J}^{M_1}_{\beta}},\] hence $crit(\pi)= \omega_1^{\bar{N}}= \rho^{\bar{N}}_{\omega}$ by elementarity of $\pi$.

But $\bar{N} | \omega_1^{\bar{N}}= N | \omega_1^{\bar{N}}$, and as 
$\bar{N} | \omega_1^{\bar{N}}$ thinks that $\omega$ is its largest cardinal, 
$N | \omega_1^{\bar{N}}$ must believe this as well. But then there can not be an extender on the $N$-sequence which is indexed at $\omega_1^{\bar{N}}$, as otherwise $N | \omega_1^{\bar{N}}$ would think that $\omega_1^{\bar{N}}$ is inaccessible, which is a contradiction.
Hence, the condition $lh(E)=\rho_{\omega}^{\bar{N}}$ is impossible and all that is left is case 1, so $\bar{N}$ is an initial segment of $M_1$.

This shows that Jensen's construction of a $\diamondsuit$-sequence succeeds when applied to $M_1$. It is straightforward to verify that the recursive construction can be carried out in $M_1 | \omega_1$ by absoluteness. Consequentially the $\diamondsuit$-sequence is a $\Sigma_1$-definable class over $M_1 | \omega_1$.

We can use the $\diamondsuit$-sequence to construct an independent $\omega$-sequence of Suslin trees
due to a result of Jensen.

\begin{definition}
 Let $T$ be a tree and $a \in T $ be a node, then $T_a$ denotes the tree $\{ x\in T \, : \, x>_Ta \}$.
 A Suslin tree $T$ is called full if for any level $\alpha$ and any finite sequence of nodes
 $a_0,...,a_n$ on the $\alpha$-th level of $T$, the tree $T_{a_0} \times T_{a_1}\times \cdot \cdot \cdot \times T_{a_n}$
 is a Suslin tree again.
\end{definition}
A proof of the next result can be found in \cite{Handbook of topology}, Theorem 6.6.
\begin{theorem}
 $\diamondsuit$ implies the existence of a full Suslin tree. Consequently if $\diamondsuit$ holds
 then there is an $\omega$-length sequence of Suslin trees $\vec{T} =\{ T_0, T_1,... \}$ such that 
 any finite product of members of $\vec{T}$ is a Suslin tree again.
\end{theorem}

The above proof, which is a refinement of Jensen's original construction of a Suslin tree in that one recursively picks (using the definable wellorder) at limit stages branches through $T$ which are generic for finite products over the least countable initial segment of $M_1$ which is able to see the construction up to that point, in fact relativizes down to $M_1 | \omega_1$, as the $\Delta^1_3$-definition of the wellorder of the $M_1$-reals can be applied inside $M_1 | \omega_1$, and the computation will always be correct. Hence, over $M_1 | \omega_1$ one can always define a full Suslin tree just as in $M_1$ and hence an $\omega$-sequence of independent Suslin trees.

The second parameter in the statement of the theorem, namely the ladder system $\vec{C}$ can be replaced by $M_1 | \omega_1$ as well, for $M_1 | \omega_1$ can compute a canonical ladder system with the help of the $M_1$-wellorder of the reals. 
\begin{theorem}
Working in $M_1$, there is a $\Sigma_1 ( \{ \omega_1 \} )$-definable $\omega$-length sequence $\vec{T}^0$ of independent Suslin trees. Also there is a $\Sigma_1 ( \{ \omega_1 \} )$-definable ladder system $\vec{C}$.
\end{theorem}
For the rest of this article, $\vec{T}^0$ and $\vec{C}$ are always as defined above.

\subsection{Definition of $W_0$}

We finally have all the ingredients to successfully define the suitable ground model $W_0$ over which coding arguments will work in a nice way.
To form $W_0$ we start with $M_1$ as our ground model and let $\delta$ be its unique Woodin cardinal. Using a $\delta$-long iteration we shall produce a generic extension $W_0$ of $M_1$ which forces with forcings of the form
\begin{enumerate}
\item $\forceP_{\beta \gamma \delta}$  for triples $\omega_1 < \beta < \gamma < \delta <\omega_2$ of uncountable cofinality.
\item $\forceP_r$ for reals $r$.
\item $\mathbb{A}_F (X)$ for $X \subset \omega_1$.
\item $\forceP_J$  which denotes Jech's forcing for adding a Suslin tree.
\item $\mathbb{S} (\vec{S} )$, for $\vec{S}$ a maximal antichain in $P(\omega_1) \slash \NS$.
\end{enumerate}
We can use either Miyamoto's nice iterations (see \cite{Miyamoto}) or revised countable support iteration due to a result of Fuchs-Switzer (see \cite{FuSw} ). Both iterations have the feature that ``Suslin tree preservation$"$ is preserved when iterating with nice iterations or RCS-iterations. This preservation is key as we will use the added Suslin trees for a second iteration over $W_0$ which will use these Suslin trees for coding purposes.

Without going into details we just state that we can adapt Shelah's proof, that a Woodin cardinal suffices to get $\NS$ saturated in this context as well. So $\NS$ will be saturated  in $W_0$ and in fact its saturation is ccc indestructible, that is further ccc forcings preserve its saturation. For details see \cite{NS Delta_1} and \cite{FH}. 

 To summarize, the main features of $W_0$ are:
\begin{enumerate}
\item $\aleph_1^{M_1} = \aleph_1^{W_0}$.
\item The Woodin cardinal $\delta$ becomes $\aleph_2$ in $W_0$.
\item $\NS$ is saturated.
\item The saturation of $\NS$ is indestructible by further ccc forcings, that is $\NS$ remains saturated in any outer model obtained with a forcing with the ccc over $W_0$.

\item There is a $\delta=\omega_2$-length, independent sequence $\vec{T}$ of $\omega_1$-Suslin trees.
\item For every subset $X$ of $\omega_1$, there is a real $r_X$ which is an almost disjoint code of $X$ with respect to the canonical $M_1$-definable, almost disjoint family of reals $D \subset \omega^{\omega} \cap M_1$.
\item Every triple of limit ordinals $(\alpha, \beta, \gamma) < \omega_2$ of uncountable cofinality is stabilized in the sense of $(\dagger)$.
\item Every real is coded by a triple of limit ordinals of uncountable cofinality $(\alpha,\beta,\gamma)$.
\end{enumerate}

We shall neither define the iteration over $M_1$ which will yield $W_0$ in detail, nor shall we prove that $W_0$ has indeed the mentioned properties and refer once more to \cite{NS Delta_1}, where everything is shown. There is a canonical, definable well-order of $P(\omega_1)$ in $W_0$.
 
  \begin{definition}\label{Definition Wellorder}
  Let $X, Y \in P(\omega_1)^{W_0}$ then let
  $X \unlhd Y$ if the antilexicographically least triple of ordinals $(\alpha_0, \beta_0, \gamma_0)$
  which code a real $r_0$ which codes $X$ with the help of the a.d. family $F$ is antilexicographically less or equal than 
  the antilexicographically least triple of ordinals $(\alpha_1, \beta_1, \gamma_1)$ which codes a real $r_1$ which in turn
  codes $Y$.

  \end{definition}

The definable wellorder $\unlhd$ of $P(\omega_1)$ unlocks a definition for a
canonical sequence of length $\omega_2$ of independent Suslin trees. The first entry of 
that sequence is, for technical reasons which will 
become clear later defined differently.
We start with our fixed independent $\omega$-sequence $\vec{T}^0$ and let $\vec{T}^{\alpha}$ be the $\unlhd$-least $\omega_1$ sequence of Suslin trees such that $\bigcup_{\beta < \alpha} \vec{T}^{\beta}$ concatenated with $\vec{T}^{\alpha}$ remains an independent sequence of Suslin trees.

As the wellorder $\unlhd$ in fact talks about the reals which are almost disjoint codes for the corresponding elements of $P(\omega_1)$ it will be useful to give that sequence of reals a name as well. For every element $X$ in $P(\omega_1)$, the set of reals which are almost disjoint codes for $X$ is infinite. In the following we nevertheless talk about \emph{the} real $r_X$ which codes $X \in P(\omega_1)$ by which we mean the $\unlhd$-least such real coding $X$.

\begin{definition}
In $W_0$, 
let $(r_i \, : \, i < \omega_2)$ be the sequence of reals defined recursively
 as follows:
 
 \begin{itemize}
  \item $r_0$ is the real which codes a subset of $\omega_1$, which codes the independent $\omega$-sequence of Suslin trees $\vec{T}^0$.
 
  \item $r_{\alpha}$, for $\alpha > 0$ is the least real which is an almost disjoint code for the $\unlhd$-least subset of $\omega_1$ which itself is a code for an $\omega_1$-sequence of independent Suslin trees $\vec{T}^{\alpha}$,
  such that the concatenated sequence of the union of the Suslin trees coded in $(r_{i} \, : \, i < \alpha)$ and $\vec{T}^{\alpha}$
  forms an independent sequence again.
 \end{itemize}
 
\end{definition}

What is very important is that this definable $\omega_2$-sequence of
independent Suslin trees will be definable in certain outer models of $W_0$.

 \begin{lemma}\label{Definability of Suslin trees}
 Suppose that $W^{\ast}$ is a set-generic, ccc extension of $W_0$.
 Then $W^{\ast}$ is still able to define the $\omega_2$-sequence of independent
 $W_0$-Suslin trees $\vec{T}$.
\end{lemma}
\begin{proof}
 
 Note first that if $r \in W^{\ast}$ is a real coded by a triple of ordinals in $(\alpha, \beta, \gamma)$ in
 $W^{\ast}$, then there is a reflecting sequence $(N_{\xi} \, : \, \xi < \omega_1)$ in
 $W^{\ast}$, $\bigcup_{\xi < \omega_1} N_{\xi} = \gamma$, such that 
 for club-many $\xi$, $r = \bigcup_{\eta \in (\nu, \xi)} s_{\alpha \beta \gamma} (N_{\eta}, N_{\xi})$.
 As $W^{\ast}$ is a ccc-extension of $W_0$, there is a reflecting sequence $(P_{\xi} \, : \, \xi < \omega_1)$ which is 
 an element in $W_0$, and such that $C:=\{ \xi < \omega_1 \, : \, P_{\xi} = N_{\xi}\}$ is 
 club containing in $W^{\ast}$. Indeed, every element of $(N_{\xi} \, : \, \xi < \omega_1)$ is a countable set of ordinals in $W^{\ast}$, thus can be covered by a countable set of ordinals from $W_0$.  As a consequence the sequence $(N_{\xi} \, : \, \xi < \omega_1)$ can be transformed into a continuous, increasing sequence $(P_{\xi} \, : \, \xi < \omega_1)$ in $W_0$ which coincides on its limit point with $(N_{\xi} \, : \, \xi < \omega_1)$, just as desired.
 
 But as ccc extensions preserve stationarity, the set $$\{ \zeta < \omega_1 \, : \, \exists \nu < \zeta \, ( \bigcup_{\eta \in (\nu, \zeta))}
 s_{\alpha \beta \gamma} (P_{\eta}, P_{\zeta}) = r )\}$$ which is an element of $W_0$ must contain a club from 
 $W_0$. Hence $r$ is coded by the triple $(\alpha, \beta, \gamma)$ already in $W_0$.
 
 As a consequence $P(\omega_1)^{W_0}$ is definable in $W^{\ast}$, it will be precisely
 the set of subsets of $\omega_1$ which have reals which code it with the help of 
 the almost disjoint family $F$, and such that these reals are themself
 coded by triples of ordinals below $\omega_2$ in the sense of $(\ddagger)$.
 
 Thus $W^{\ast}$ can define $W_0$-Suslin trees and our wellorder $<$ on $P(\omega_1)^{W_0}$, hence will be able to define 
 the $\omega_2$-sequence of independent Suslin trees of $W_0$.
 
\end{proof}

\section{Towards a proof of the theorem}
In this section we shall prove the main theorem of this work. We aim to prove it starting from our ground model $W_0$.

We shall sketch, omitting a lot of technical issues, a simplified idea of the proof of the theorem first.

In a finitely supported iteration over $W_0$ we will use the definable list of independent Suslin trees to code up all the characteristic functions of stationary subsets of $\aleph_1$ with the help of branching or specializing elements of the list. The fact that $H(\omega_2)^{W_0}$ is $\Sigma_1$-definable in all ccc extensions of $W_0$ can be utilized to define (with a $\Sigma_1$-definition) a class of suitable, $\aleph_1$-sized models of fragments of $\ZFC$ which are sufficient to define the list of Suslin trees properly and consequently enable a boldface $\Sigma_1$ definition of stationarity. Indeed the assertion "$S$ is stationary" will be equivalent to the statement that there is a suitable model which defines some $\omega_1$ block of independent Suslin trees and this block has a pattern of branches and specializing functions which correspond to the characteristic function of $S$.

\subsection{The second iteration.}
\subsubsection{An outline of the idea}
Let us quickly describe the situation we are in. We have obtained a model $W_0=M_1[G]$ with the following
properties:
\begin{enumerate}
\item $\aleph_1^{M_1} = \aleph_1^{W_0}$.
\item The Woodin cardinal $\delta$ becomes $\aleph_2$ in $W_0$.
\item $\NS$ is saturated.
\item The saturation of $\NS$ is indestructible by further ccc forcings, that is $\NS$ remains saturated in any outer model obtained with a forcing with the ccc over $W_0$.

\item There is a $\delta=\omega_2$-length, independent sequence $\vec{T}$ of $\omega_1$-Suslin trees which is $\Sigma_1 ( \{ \omega_1 \} )$-definable over $W_0$.
\item For every subset $X$ of $\omega_1$, there is a real $r_X$ which is an almost disjoint code of $X$ with respect to the canonical $M_1$-definable, almost disjoint family of reals $D \subset \omega^{\omega} \cap M_1$.
\item Every triple of limit ordinals $(\alpha, \beta, \gamma) < \omega_2$ of uncountable cofinality is stabilized in the sense of $(\dagger)$.
\item Every real is coded by a triple of limit ordinals of uncountable cofinality $(\alpha,\beta,\gamma)$.
\end{enumerate}
The patient reader will notice that we have not touched the issue of the definabilty of the nonstationary ideal or the $\Delta^1_4$-definable well-order of the reals. The second iteration is entirely concerned with coding these two family of sets. We will use the fact that a Suslin tree $T$ can be destroyed in two mutually exclusive ways using forcings with the countable chain condition: either we add a branch to $T$ or an uncountable antichain. This enables us to write the characteristic function reals or of stationary subsets of $\omega_1$ into the definable sequence $\vec{T}$. The fact that $\vec{T}$ is an independent sequence has as a consequence that the destruction of fixed elements of $\vec{T}$ will not affect the Suslinity of the other elements of $\vec{T}$.

Thus the following strategy is promising. We start an $\omega_2$-length iteration of forcings which either specialize or shoot a branch through elements of $\vec{T}$. The iteration uses finite support.
We will enumerate in an $\omega_2$-length list all the right pairs of reals, which eventually should be all the pairs of our to be defined well-order and the stationary subsets of $\omega_1$, pick the sets listed (given by some fixed bookkeeping function) one after another and code the characteristic function of it into the  $\omega_1$-blocks of our definable list of Suslin trees $\vec{T}$. 
This coding will create new stationary subsets which we list again and code up using fresh elements of $\vec{T}$ which we have not destroyed yet. Bookkeeping will yield that after $\omega_2$-many stages we will catch our tail.

The result will be a generic extension $W_0[H]$ of $W_0$ by a forcing which has the countable chain condition. 
$W_0[H]$ can define its stationary subsets of $\omega_1$ in a new way: $S \subset \omega_1$ is stationary if and only if there is an $\omega_1$-block in $\vec{T}$ (note that $\vec{T}$ is still definable in $W_0[H]$ by Lemma \ref{Definability of Suslin trees}) such that the characteristic function of $S$ is written into this $\omega_1$-block of elements of $\vec{T}$. $W_0 [H]$ can also define a well-order $<$ on its reals via $x<y$ if there is an $\omega$-block of trees from $\vec{T}$ such that the characteristic function of the pair $(x,y)$ is coded into it.

A calculation yields that this new definition of stationarity is boldface $\Sigma_2$ over the $H(\omega_2)$ of $W_0[H]$, and the definition of $<$ is of the same complexity, thus it seems like we have not gained anything substantial. The next idea is to force over $W_0$ to obtain a universe $W_1$ which will have an easy definition for $\aleph_1$-sized, transitive models of $\ZFP$ which are sufficiently clever to correctly define the sequence $\vec{T}$ up to their ordinal height, plus are correct about the predicate ``$S$ is stationary$"$. This is nontrivial, as $\aleph_1$-sized models typically will not be correct when defining $\vec{T}$. 

Working over $W_1$  it will be possible to alter the above iteration which makes a simple definition of a class of $\aleph_1$-sized, so-called suitable models possible, which can be utilized to read off the created information on the sequence of Suslin trees. This will buy us one quantifier and we eventually arrive at a $\Sigma_1 (\vec{C}, \vec{T}^0)$-predicate for stationarity. Further refinements will yield that one can add additional almost disjoint coding forcings of carefully defined subsets of $\omega_1$ to obtain a $\Sigma^1_4$-definition of the well-order.

\subsubsection{Suitability}
We begin to define thoroughly how the second iteration which will yield $W_1$, using $W_0$ as a ground model does look like. 
We already hinted that, in order to lower the complexity of 
a description of stationarity we need a new notion for suitable models which
will be able to define the sequence of Suslin trees $\vec{T}$ correctly.
With the notion of suitability it will become possible 
to witness stationarity already in $\aleph_1$-sized $\ZFP$ models,
as we shall see soon.
\begin{definition}
 Let $M$ be a transitive model of $\ZFP$ of size $\aleph_1$. We say that $M$ is pre-suitable if it satisfies the 
 following list of properties:
 \begin{enumerate}
  \item $M_1 | \omega_1 \subset M$ so in particular our distinguished $\omega$-sequence of independent Suslin tree $\vec{T}^0$ and our distinguished ladder system $\vec{C}$ are both in $M$.
  \item $\aleph_1$ is the biggest cardinal in $M$ and $M \models \forall x (|x|\le \aleph_1)$.
  \item Every set in $M$ has a real in $M$ which codes it in the sense of almost disjoint coding
  relative to the fixed family of almost disjoint reals $F$.
  \item Every real in $M$ is coded by a triple of ordinals in $M$, i.e. 
  if $r \in M$ then there is a triple $(\alpha, \beta, \gamma) \in M$ and a reflecting
  sequence $(N_{\xi} \, : \, \xi < \omega_1)\in M$ which code $r$.
  \item Every triple of ordinals in $M$ is stabilized in $M$: for $(\alpha, \beta, \gamma)$
  there is a reflecting sequence $(P_{\xi} \, : \, \xi < \omega_1) \in M$ which witnesses that
  $(\alpha, \beta, \gamma)$ is stabilized.
 \end{enumerate}

\end{definition}

Note that the statement '``$M$ is a pre-suitable model$"$ is completely internal in $M$ and hence a $\Sigma_1(\vec{C}, \vec{T}^0)$ and in particular a $\Sigma_1 (\{\omega_1 \} )$-formula.
Further note that by the proof of Lemma \ref{Definability of Suslin trees}, if $W^{\ast}$ is a ccc extension of $W_0$ and $M$ is a pre-suitable
model in $W^{\ast}$ then $M \subset W_0$, as ccc extensions will not add new reflecting sequences.
\begin{definition}
 Let $M$ be a pre-suitable model. We say that $M$ is $W_0$-absolute for Susliness
 if $T \in M$ is an element from $W_0$ and $M \models T \text{ is Suslin}$, then $T$ is Suslin in $W_0$.
 Likewise we say that $M$ is $W_0$-absolute for stationarity if
 $S \in M$ and $M$ thinks that
 $S$ is a stationary subset of $\omega_1$ then $S$ is a stationary subset of $\omega_1$ in $W_0$. 
 A pre-suitable model which is $W_0$-absolute for stationarity and Susliness is called suitable.
\end{definition}

We have already seen in Lemma \ref{Definability of Suslin trees} that ccc extensions of $W_0$ will still be able to define our
$\omega_2$-sequence of independent $W_0$-Suslin trees $\vec{T}$. With the notion of suitability
we can localize this property in the following sense:
\begin{lemma}
 Let $W^{\ast}$ be a ccc extension of $W_1$, and let $M \in W^{\ast}$ be a suitable model.
 If $M$ computes the $\omega_2$-length sequence of independent Suslin trees from $W_0$ using its local wellorder $\unlhd_M$,
 then the computation will be correct, i.e. $\vec{T}^M = \vec{T} \cap M$.
\end{lemma}
\begin{proof}
 We shall show inductively that for every $\eta \in M$, the $\eta$-th
 block of $\vec{T}$, $\vec{T}^{\eta}$ will be computed correctly by $M$ (if an element of $M$).
 For $\eta=0$ this is true as $\vec{T}^0$ is by definition of suitability an element of $M$.
 Let $\eta \in M$ and assume by induction that the sequence $\vec{T}$ up to the $\eta$-th block is computed correctly by $M$. We shall show that $M$ computes the $\eta+1$-th block correctly. Recall that $\vec{T}^{\eta+1}$ was defined to be the $\unlhd$-least $\omega_1$-block of independent Suslin trees such that $\bigcup_{\beta \le\eta} \vec{T}^{\beta}$ concatenated with $\vec{T}^{\eta+1}$ remains an independent sequence in $W_0$. 
 
 Assume for a contradiction that the suitable $M$ computes $\vec{T}'$ as its own different version of $\vec{T}^{\eta+1}$. Thus there is a real
 $r'$ which codes $\vec{T}'$, and $r'$ itself is coded into a triple of $M$-ordinals $(\alpha', \beta', \gamma')< \omega_2^{M}$. Let $r_{\eta+1}$ be the real which codes $\vec{T}^{\eta+1}$.
 We claim that $r_{\eta+1} < r'$ by which we mean that the least triple of ordinals which codes $r_{\eta+1}$ is antilexicographically less than the least triple which codes $r'$. Otherwise $r'< r_{\eta+1}$ and by the Suslin-absoluteness of $M$
 the independent-$M$-Suslin trees coded into $r'$ would be an independent $\omega_1$-sequence of Suslin trees in $W_0$,
 moreover they would still form an independent sequence when concatenated with $\vec{T}^{\eta}$ in $W_0$ which
 is a contradiction to the way $\vec{T}^{\eta+1}$ was defined.
 
 So $r_{\eta+1}<r'$, so the least triple of ordinals $(\alpha, \beta, \gamma)$ coding $r_{\eta+1}$
 is antilexicographically less than $(\alpha', \beta', \gamma')$. Note that the suitability of $M$ implies
 that $(\alpha, \beta, \gamma)$ is stabilized in $M$. Thus there is a reflecting sequence
 $(P_{\xi} \, : \, \xi < \omega_1)$ in $M$ witnessing this. As $W^{\ast}$ is a ccc extension of $W_0$ we can assume
 that the sequence is in fact an element of $W_0$. At the same time there is a reflecting 
 sequence $(N_{\xi} \, : \, \xi < \omega_1)$ in $W_0$ which witnesses that $r_{\eta+1}$ is coded by $(\alpha, \beta, \gamma)$.
 By the continuity of both sequences, there is a club $C$ in $W_0$ such that
 $\forall \xi \in C ( N_{\xi} = P_{\xi}).$ Thus the limit points of $C$ witness that
 in fact the sequence $(P_{\xi} \, : \, \xi < \omega_1)\in M$ codes $r_{\eta+1}$ as well but the club
 $C$ is in $W_0$, so we need an additional argument to finish. Recall that the suitability of $M$
 implies that $M$ is absolute for stationarity, thus if the set 
 $\{ \xi < \omega_1 \, : \, \exists \nu < \xi ( \bigcup_{\zeta \in (\nu, \xi)} s_{\alpha \beta \gamma} (P_{\zeta}, P_{\xi}) \ne r_{\eta+1}) \}$
 would be stationary in $M$ it would be stationary in $W_0$ which is a contradiction.
 So $M$ computes $r_{\eta+1}$ correctly and the rest of the inductive argument can be repeated exactly as 
 above to show that $\vec{T}^M = \vec{T} \upharpoonright (M \cap Ord)$ as desired.
 
  \end{proof}

So suitable models will compute $\vec{T}$ correctly, and if we start to write information into $\vec{T}$, a suitable model can be used to read it off. In $W_0$ cofinally many suitable model below $\omega_2$ exist.

\begin{lemma}
Work in $W_0$. If we set
$$ P:= \{ \eta < \omega_2 \, : \, \exists M (M \text{ is suitable } \land M \cap Ord = \eta\}$$ then $P$ is unbounded in $\omega_2$.
\end{lemma}
\begin{proof}
Recall the iteration $(\forceP_{\alpha} \, : \, \alpha < \delta)$ to force $W_0$. Whenever we are at an intermediate stage $\kappa < \delta$ such that $\kappa$ is Mahlo then, if $G_{\kappa}$ denotes the generic filter for $\forceP_{\kappa}$, $H(\omega_2)^{V[G_{\kappa}]}$ will be a pre-suitable model.
Moreover, by the properties of nicely supported iterations, stationary subsets of $\omega_1$ in $H(\omega_2)^{V[G_{\kappa}]}$ will remain stationary in $W_0$ and Suslin trees in $H(\omega_2)^{V[G_{\kappa}]}$ will remain Suslin trees in $W_0$, as the tail iteration $\forceP_{[\kappa, \delta)}$ is a nicely supported iteration of semiproper, Suslin tree preserving notions of forcing over the ground model $V[G_{\kappa}]$. Thus every  $H(\omega_2)^{V[G_{\kappa}]}$, for $\kappa$ Mahlo, is a suitable model which gives the assertion of the Lemma.

\end{proof}

We shall use forcing to obtain a universe in which the just defined set of suitable models 
\begin{itemize}
\item[$U'$]
$:=\{ M \, : \, \exists \kappa < \delta (\kappa \text{ is Mahlo } \land M=H(\omega_2)^{V[G_{\kappa}]} \}$
\end{itemize} 
becomes easily definable. 
This will be our first forcing in order to obtain $W_1$ over $W_0$. We note that every $M \in U'$ is itself coded by a real relative to the almost disjoint family $F$. We want to use  the variant of almost disjoint coding forcing to code up  \begin{itemize}
\item[$U$] $:= \{ r_M \, : \, r_M $ is the least almost disjoint code for a subset of $\omega_1$ which codes $M \in U' \}$
\end{itemize}
into one real $r_{U}$. Recall that $\mathbb{A}(U)$ is a forcing of size $|U|=\delta$ which is Knaster and which adds a real $r_U$ such that in $W_0^{\mathbb{A}(U)}$ the following holds:
\begin{itemize}
\item[$(\ast)$] $\forall x \in 2^{\omega} \cap W_0 \,( x \in U \leftrightarrow r_U \cap S(x)$ is finite$)$.
\end{itemize}
In a next step we will code the characteristic function of the real $r_U$ into a pattern on $\vec{T}^0$. We use the fact that a Suslin tree can generically be destroyed in two mutually exclusive ways. We can generically add a branch or generically add an antichain without adding a branch.
We fix the first $\omega$-block of independent Suslin trees $\vec{T}^0$ and we let 
$\mathbb{D} (r_U) := \prod_{n \in \omega} \forceP_n$ with finite support, where
\begin{equation*}
\forceP_n = \begin{cases}
T^{0}_n  & \text{ if } r_U(n)=1 \\ Sp(T^{0}_n) & \text{ if } r_U(n)=0
\end{cases}
\end{equation*}
and $T^{0}_n$ denotes the forcing notion one obtains when forcing with nodes of $T^{0}_n$ as conditions, and
$Sp(T^{0}_n)$ denotes Baumgartner's forcing which specializes $T^{0}_n$ with finite conditions and which is known to be ccc (see \cite{Baumgartner}).
Iterations of the just described form always have the countable chain condition.

\begin{lemma}\label{ccc Coding}
Let $\vec{T}= (T_{\alpha} \, : \, \alpha < \eta)$ be an independent sequence of Suslin trees of length $\eta$. Let $f: \eta \rightarrow 2$ be an arbitrary function and let $T_{\alpha}$ also denote the partial order when forcing with the tree $T_{\alpha}$ and let $Sp(T_{\alpha})$ be the forcing which specializes the tree $T_{\alpha}$.

Then if we consider the finitely supported product $\mathbb{D} (f) := \prod_{\beta < \eta} \forceP_{\beta}$ where
then $\mathbb{D}$ has the countable chain condition.
\end{lemma}
\begin{proof}
Fix an arbitrary $f: \eta \rightarrow 2$. We prove the Lemma using induction over the length $\eta$. The limit case is true as we use finite support. Thus assume the assertion of the Lemma is true for $\eta$ and we want to show it is true for products of length $\eta+1$. Assume for a contradiction that $\prod_{\alpha < \eta+1}\forceP_{\alpha}$ (according to $f$) does not have the countable chain condition.
Hence the tree $T_{\eta}$ is not a Suslin tree in $V^{\prod_{\alpha < \eta} \forceP_{\alpha}}$, as otherwise both forcings $T_{\eta}$ and $Sp(T_{\eta})$ would have the countable chain condition.
But $\prod_{\alpha < \eta} \forceP_{\alpha} \ast T_{\eta} = \prod_{\alpha < \eta} \forceP_{\alpha} \times T_{\eta} = T_{\eta} \times \prod_{\alpha < \eta} \forceP_{\alpha}$, and the latter is a forcing with the countable chain condition. Indeed as $T_{\eta}$ does not touch the Susliness of any member of $\vec{T}$ besides $T_{\eta}$,  $(T_{\alpha} \, : \, \alpha < \eta)$ is an independent sequence of Suslin trees in $V^{T_{\eta}}$, thus
by induction hypothesis,  $\prod_{\alpha < \eta} \forceP_{\alpha}$ has the countable chain condition in $V^{T_{\eta}}$, so it has the countable chain condition in $V$.
\end{proof}

In particular this means that $\mathbb{D} (r_U)$ as defined above is a forcing with the countable chain condition which writes the real $r_U$ into a pattern of 0 and 1's on the sequence $\vec{T}^0$ of Suslin trees. We
let $H_0$ be
a generic filter for $\mathbb{A}_U$ over $W_0$ and
let
 $H_1$ denote a $W_0[H_0]$-generic filter for $\mathbb{D} (r_U)$. The resulting model $W_1:=W_0[H_0][H_1]$ is the ground model for a second iteration we define later.
 
\begin{lemma}
 Let $W_1$ be the universe $W_0[H_0][H_1]$

 Then
 $W_1:= W_0[H_0][H_1]$ is a ccc extension of $W_0$ which satisfies:
 \begin{itemize}
 \item $ r_U(n)=1$ if and only if $T^{0}_n$ has a branch.
 \item $ r_U(n)=0$ if and only if $T^{0}_n$ is special.
\end{itemize}
\end{lemma}

\begin{proof}
It suffices to show that $W_0[H_0]$ is a Suslin-tree-preserving extension of $W_0$. This is clear as $\mathbb{A}(U)$ is Knaster. 

\end{proof}
To summarize we arrived at a situation where the real $r_U$, which captures all the information about a set of suitable models, is written into the sequence of Suslin trees $\vec{T}^0$. Consequentially, any transitive, $\aleph_1$-sized model  $M$ which contains $\vec{C}$, $\vec{T}^0$ and which sees that every tree in $\vec{T}$ is destroyed can compute the real $r_U$ and thus has access to a set of suitable models, which in turn can be used to compute $\vec{T}$ inside $M$ in a correct way. This line of reasoning remains sound in all outer ccc extensions of $W_1$ as we shall see. Thus the ability of finding $\vec{T}$ in suitable models gives rise to the possibility of using $\vec{T}$ for additional coding arguments over $W_1$.

\begin{lemma}\label{sigma_1 set of suitable models}
 Let $W^{\ast}$ be a ccc extension of $W_1$ then there is a $\Sigma_1(\vec{C}, \vec{T}^0)$-formula $\Phi(v)$ such that whenever $x \in 2^{\omega}$ and $W^{\ast} \models \Phi(x)$ then $x$ is the almost disjoint code for a suitable model. 
\end{lemma}
\begin{proof}
The formula $\Phi(x)$, is defined as follows:

 \begin{itemize}
 \item[$\Phi(x)$]   if and only if there is a transitive, $\aleph_1$-sized $\ZFP$ model $N$ which contains $\vec{C}$ and $\vec{T}^0$ such that the following holds in $N$:  
   \begin{itemize}
    \item $N$ sees a full pattern on $\vec{T}^0$, i.e. for every $n \in \omega$ and 
   every member $T^0_{n}$, $N$ has either a branch through $T^0_{n}$ 
   or a function which specializes $T^0_{n}$. 
\item  The pattern on $\vec{T}^0$ corresponds to the characteristic function of a real $r$ and $N$ thinks that there is a pre-suitable $N'$ such that $x \in N'$ and $(S(x) \cap r)$ is a finite set.
 \end{itemize}
   \end{itemize}
This is a $\Sigma_1(\vec{C}, \vec{T}^0)$-formula, as it is of the form $\exists N( N\models ...)$. 
We shall show that whenever $x$ is a real from $W^{\ast}$ such that $\Phi(x)$ holds then $x$ is the almost disjoint code for a suitable model.
Recall the set $U'$ and $U$ we defined above.
Note first that $N$ has access to the set $U$ which is the set of reals which are codes for the set of suitable models $U'$.
This is clear as $\vec{T}^0 \in N$, thus if $N$ sees a pattern on $\vec{T}^0$, this pattern must be the unique pattern from $W_1$, which corresponds to the characteristic function of the real $r_U$. The statement "$N'$ is a pre-suitable model" is a $\Delta_1(\vec{C})$ formula in $N'$, thus absolute for the transitive $N$. So if $N$ thinks that there is a pre-suitable $N'$ which contains $x$ as an element, then this is true in $W^{\ast}$. As $W^{\ast}$ is a ccc extension of $W_1$ and hence of $W_0$ as well, we know already that we can express the statement "$y \in 2^{\omega} \cap W_0$" in $W^{\ast}$ as "$\exists P (P$ is pre-suitable and $y \in P)$". Thus if $N$ thinks that there is a pre-suitable $N'$ which contains $x$ as an element, then $x \in W_0$, and now $(\ast)$ from above applies to conclude that indeed $x \in U$ which is what we wanted.
\end{proof}
\subsubsection{Coding machinery}
We work now over $W_1$ and start a finite support iteration $(\forceP_{\alpha},\forceQ_{\alpha} \, : \, \alpha < \omega_2)$ of length $\omega_2=\delta$. Our goal is to define a $\Sigma^1_4$-predicate $\sigma(x)$ such that
\begin{enumerate}
\item $W_1 \models \lnot \exists x \sigma (x)$.
\item For an arbitrary real $x$, there is a forcing $\operatorname{Code} (x)$ which has size $\aleph_1$ and has the ccc and which has the effect that $W_1^{\operatorname{Code} (x)} \models \sigma (x)$, yet for every real $y \ne x$, 
$W_1^{\operatorname{Code} (x)} \models \lnot \sigma (y)$
\end{enumerate}

We shall define the coding forcing $\operatorname{Code} (x)$ along our desired $\Sigma^1_4$-formula $\sigma$. Assume that $x \in 2^{\omega}$ is an arbitrary real in $W_1$.
In a first step we fix the first $\omega$-block $\vec{T}^1$ from $\vec{T}$ and use the same technique as in Lemma \ref{ccc Coding} to write $x$ into $\vec{T}^1$, i.e. we first let the first factor of $\operatorname{Code}(x)$ be the finitely supported product of factors defined as follows:
\begin{equation*}
\forceP_{n } = \begin{cases}
T^1_{n}  & \text{ if } x(n)=1 \\ Sp(T^1_{n}) & \text{ if } x(n)=0
\end{cases}
\end{equation*}

To form the second forcing we need to find a ``nice$"$ subset of $\omega_1$ which codes all the branches and specializing functions along with some suitable models.
Let $A$ be a suitable model which computes $\vec{T}^1$ correctly.
We collect the set $M_1 | \omega_1$, the relevant branches and specializing functions through elements of $\vec{T}^0$ which create the pattern which codes up $r_U$, the relevant reflecting sequences to define the suitable $A$ and write everything into one set $X \subset \omega_1$. 
Note that i if $L_{\zeta}[X]$ is a $\ZFP$-model which contains $X \subset \omega_1$, we obtain that
\begin{align*}
L_{\zeta} [X] \models &n \in x \rightarrow T^1_{n} \text{ has an $\omega_1$-branch and } \\& 
n \notin x \rightarrow T^1_{n}  \text{ has a specializing function}
\end{align*}
where we define $\vec{T}^1$ inside $A \in L_{\zeta} [X]$ and $A$'s suitability is confirmed via $M_1 | \omega_1$ and the $M_1 | \omega_1$-definable trees $\vec{T}^0$.

We first note that any transitive, $\aleph_1$-sized $\ZFP$ model $M$ which contains $X$ will satisfy
\begin{align*}
(M, \in,  \mathcal{J}^{M_1}_{\omega_1}) \models &``\text{Decoding X yields several objects namely} \\& \text{ a model $m$ and $m=\mathcal{J}^{M_1}_{\omega_1}$}=\bigcup_{\mathcal{J}^{M_1}_{\eta} \in \mathcal{I}} \mathcal{J}^{M_1}_{\eta},  \\ & \text{ a suitable model $a$ which in turn computes trees $\vec{t^1}$ } \\& \text{some branches $\vec{b}$ and specializing functions $\vec{f}$ through $\vec{t^1}$},\\& \text{ such that for any $\eta< \omega_2$ such that $L_{\eta}[a][\vec{b}] [\vec{f}] \models \ZFP$} 	 \\&
			L_{\eta}[a][\vec{b}] [\vec{f}] \models  \forall n  ( n \in x \rightarrow T^1_{n}  \text{ has an $\omega_1$-branch}  \\&
		\qquad \qquad  	n \notin x \rightarrow T^1_{n}  \text{ has a specializing function})
			\end{align*}
In particular, this will be true for a $\ZFP +`` \aleph_1$ exists$"$ model of the form $(L_{\xi}[X],\in, \mathcal{J}^{M_1}_{\omega_1})$, $\xi < \aleph_2$.
If we consider the club \[ C:= \{ \eta < \omega_1 \, : \, \exists (M, \in, P) \prec (L_{\xi}[X],\in,\mathcal{J}^{M_1}_{\omega_1})( |M|=\aleph_0 \land \eta= \omega_1 \cap M) \}\]
then if $(N, \in)$ is an arbitrary countable transitive model of $\ZFP$ such that
$X\cap \omega_1^N \in N$ and $\omega_1^N  \in C$, then $N$ will decode out of $X \cap \omega_1^N$ exactly what $(\bar{M},\in,\mathcal{J}^{M_1}_{\eta})$ decodes out of $X \cap \omega_1^N$, where the latter is the transitive collapse of  $(M,\in, P) \prec (L_{\xi}[X],\in, \mathcal{J}^{M_1}_{\omega_1})$, where  $X \in M,  M \cap \omega_1= \omega_1^N$. In particular, if we denote the $\Delta_1$-definable decoding functions with $dec_1,dec_2$, $dec_3$ and $dec_4$ respectively, then we obtain 
\begin{align*}
N \models \exists m_1 \, \exists \vec{c} \, \exists \vec{b}, \vec{f} (& dec_1(X\cap \omega_1^N)=m_1 \land dec_2(X \cap \omega_1^N)= a \\& \land dec_3(X \cap \omega_1^N)=\vec{b} \land dec_4(X\cap \omega_1^N)=\vec{f} \\& \text{ and for any $\ZFP$ model of the form
$L_{\eta}[X \cap \omega_1^N]$ we have that} \\&
			L_{\eta} [X \cap \omega_1^N] \models  \forall n ( n \in x\rightarrow T^1_{n}  \text{ has an $\omega_1$-branch}  \\&
		 \qquad \qquad \qquad  \quad \quad  	n \notin x\rightarrow T^1_{n}  \text{ is special})) ).
\end{align*}
Further, as $dec_1(X \cap \omega_1^N)=m_1=\mathcal{J}^{M_1}_{\eta}$, we get that
\[ dec_1(X \cap \omega_1^N) \in \mathcal{I}.\]

Now let the set $Y\subset \omega_1$ code the pair $(C, X)$ such that the odd entries of $Y$ should code $X$ and if $Y_0:=E(Y)$ where the latter is the set of even entries of $Y$ and $\{c_{\alpha} \, : \, \alpha < \omega_1\}$ is the enumeration of $C$ then
\begin{enumerate}
\item $E(Y) \cap \omega$ codes a well-ordering of type $c_0$.
\item $E(Y) \cap [\omega, c_0) = \emptyset$.
\item For all $\beta$, $E(Y) \cap [c_{\beta}, c_{\beta} + \omega)$ codes a well-ordering of type $c_{\beta+1}$.
\item For all $\beta$, $E(Y) \cap [c_{\beta}+\omega, c_{\beta+1})= \emptyset$.
\end{enumerate}
We obtain a version of the  which works for suitable countable transitive models:
\begin{itemize}
\item[] Let $M$ be an arbitrary countable transitive model of $\ZFP + `` \aleph_1$ exists$"$ for which 
there is a $\mathcal{J}^{M_1}_{\eta} \in \mathcal{I}$ such that $\omega_1^M=\omega_1^{\mathcal{J}^{M_1}_{\eta}}$ and $\mathcal{J}^{M_1}_{\eta} \in M$. Assume that $Y \cap \omega_1^M \in M$ then $M$ can decode out of  $Y \cap \omega_1$, 
\begin{itemize}
\item a model $m$,
\item a model $a$ which, $M$ believes, is suitable, and which computes $\vec{t}^1$,
\item a set of branches $\vec{b}$ through $\vec{t}^1$.
\item and specializing functions $\vec{f}$ through elements of $\vec{t}^1$ such that for any $\ZFP+ ``\aleph_1$ exists$"$-model of the form $L_{\zeta} [a,\vec{b},\vec{f}]$:
\begin{align*}
L_{\zeta} [a,\vec{b},\vec{f} \models &\forall n ( n \in x \rightarrow t^1_{n}  \text{ has an $\omega_1$-branch}  \\&
		 \qquad   	n \notin x \rightarrow t^1_{n}  \text{ is special})) ).
\end{align*}

\end{itemize}
Moreover $m$ is an $M_1$ initial segment as seen from the outside, i.e. $m = \mathcal{J}^{M_1}_{\eta} \in \mathcal{I}$.
\end{itemize}
In the last step, we use almost disjoint coding forcing relative to the $\mathcal{J}^{M_1}_{\omega_1}$-definable almost disjoint family of reals to obtain a real $r_Y$ which codes our set $Y \subset \omega_1$. Thus we obtain the following formula $\psi(x,r_Y)$ holds, where $\psi(x, r_Y)$ is defined to be the formula from above:
\begin{itemize}
\item[] Assume $M$ is a countable transitive model of $\ZFP+ ``\aleph_1$ exists$"$, and $r_Y \in M$. Assume that for $M$ 
there is a $\mathcal{J}^{M_1}_{\eta} \in \mathcal{I}$ such that $\omega_1^M=\omega_1^{\mathcal{J}^{M_1}_{\eta}}$ and $\mathcal{J}^{M_1}_{\eta} \in M$.  Assume further that $r_Y \in M$ then $M$, relative to the a.d. family of reals from $\mathcal{J}^{M_1}_{\eta}$, can decode out of  $r_Y$ the following
\begin{itemize}
\item a transitive model $m$ of $\ZFP$ and ``$\aleph_1$ exists$"$,
\item a model $a$ which, $M$ believes, is suitable, and which computes $\vec{t}^1$,
\item a set of branches $\vec{b}$ through $\vec{t}^1$.
\item and specializing functions $\vec{f}$ through elements of $\vec{t}^1$ such that for any $\ZFP+ ``\aleph_1$ exists$"$-model of the form $L_{\zeta} [a,\vec{b},\vec{f}]$ (computed in $m$):
\begin{align*}
L_{\zeta} [a,\vec{b},\vec{f} \models &\forall n ( n \in x \rightarrow t^1_{n}  \text{ has an $\omega_1$-branch}  \\&
		 \qquad   	n \notin x \rightarrow t^1_{n}  \text{ is special})) ).
\end{align*}

\end{itemize}
Moreover $m$ is an $M_1$ initial segment as seen from the outside, i.e. $m = \mathcal{J}^{M_1}_{\eta} \in \mathcal{I}$.
\end{itemize}
A straightforward calculation shows that the statement $\psi(x,r_Y)$ is of the form $ (\Sigma^1_3 \rightarrow \Pi^1_3)$, thus it is a $\Pi^1_3$-formula, and stating the existence of such a real $r_Y$ is $\Sigma^1_4$ and results in our desired $\Sigma^1_4$-formula $\sigma$:
\[ \sigma(x) \equiv \exists r \psi (x,r). \]

The existence of a real $r$ witnessing $\psi(x,r)$ is sufficient to conclude that $L[r]$ contains branches through $\aleph_1$-many trees from $\vec{S}$.
\begin{lemma}\label{definable well-order of reals determines real world}
Let $r$ be such that $\psi(x,r)$ is true.
Then, working inside $L[r]$, there is a suitable $A$ such that $A$ computes $\vec{T}^1$ and there are branches $\vec{b}$ and specializing functions $\vec{f}$ such that
\begin{align*}
\forall n (&n \in x \rightarrow T^1_{n} \text{ has an $\omega_1$-branch} \\& 
n \notin x \rightarrow T^1_n \text{ is special} )
\end{align*}
\end{lemma}
\begin{proof}
We note first that $\psi(x,r)$ must also be true (ignoring its statements involving $\mathcal{I}$) for models of uncountable size where we replace $\mathcal{J}^{M_1}_{\eta}$, in $\psi$, with $\mathcal{J}^{M_1}_{\omega_1}$. Indeed, if $M$ would be an uncountable, transitive model containing $r$ and $\mathcal{J}^{M_1}_{\omega_1}$ for which $\psi(x,r)$ is false, then we let $\bar{N}$ be the transitive collapse of $N \prec M$, $r, \mathcal{J}^{M_1}_{\eta} \in N$ and $\bar{N}$ would reject $\psi(x,r)$ as well, even though $\bar{N}$ is of the right form, which gives us a contradiction.

But if $\psi(x,r)$ holds for arbitrarily large models $M$, it must be true in the universe $L[r]$. Indeed if some $\aleph_1$-sized  $\ZFP$-model of the form $L_{\zeta} [M,A, \vec{B}, \vec{F}]$, where $M,A, \vec{B}, \vec{F}$ are just the unions of the computations of $m,a, \vec{b}$ and $\vec{f}$ in suitable countable transitive models of increasing (with limit $\omega_1$) ordinal height, then first note that $M=M_1 | \omega_1$ and $L_{\zeta} [M,A, \vec{B}, \vec{F}]$ sees that there are branches $\vec{B}$ and specializing functions $\vec{F}$ such that 
\begin{align*}
 L_{\zeta} [M,A, \vec{B}, \vec{F}] \models  &\forall  n (n  \in x  \rightarrow T^1_{n}  \text{ has an $\omega_1$-branch}  \land \\&
		 \quad  	n \notin x \rightarrow T^1_{n}  \text{ is special})) .
\end{align*}
and $L_{\zeta} [M,A, \vec{B}, \vec{F}] $'s computation of $\vec{T^1}$ must be correct. As the existence of an $\omega_1$-branch through $T^1_{n}$, and the existence of a specializing function for $T^1_n$ is upwards absolute we obtain the assertion of the lemma.
\end{proof}

As an immediate consequence, the truth of the $\Sigma^1_4$-formula $\sigma$ affects the surrounding universe, in that whenever $\sigma (x)$ is true, one really will find the characteristic function of $x$ written into $\vec{T}^1$ in the real world.

We note that we defined $\operatorname{Code} (x)$ with respect to $\vec{T}^1$ but this choice was not necessary and we could have defined $\operatorname{Code} (x)$ also with respect to $\vec{T}^{\alpha}$ for any $\alpha< \omega_2$. We define
\[ \operatorname{Code} (x,\eta) := ( \prod_{n  \in x}  T^{\eta}_{n} \times \prod_{n \notin x} Sp (T^{\eta}_n ) ) \ast \mathbb{A} (Y) \]
that is $\operatorname{Code} (x,\eta)$ is the coding forcing just defined, but we use the $\eta$-th $\omega$-block of $\vec{T}$ for coding and not $\vec{T}^1$.

So to summarize our discussion so far, if we let $W^{\ast}$ be our ground model, which is defined as reproducing the move from $L$ to $W$ with $M_1$ as starting point, then there is a way of coding arbitrary reals $x$  into the $\vec{S}$-sequence, and the statement ``$x$ is coded into $\vec{S}"$ is $\Sigma^1_4(x)$.

\subsection{Definition of the final iteration}

We finally have assembled all the necessary tools to define the final iteration which will yield a generic extension $W_2$ of $W_1$ with a $\Delta^1_4$-definable well-order of the reals, a $\Delta_1  ( \{ \omega_1 \} )$-definition of $\NS$ and where, additionally $\NS$ is saturated.

Our ground model is $W_1$.  We pick the G\"odel pairing function as our bookkeeping function $F: \omega_2 \times \omega_2 \rightarrow \omega_2$. Recall that $F$ is a bijection with the additional property that $\forall \beta < \omega_2( \beta \ge max((\beta_1, \beta_2))= F^{-1} (\beta)$. The iteration is defined recursively as follows. Assume we are at stage $\alpha < \omega_2$ and we have already defined $\forceQ_{\alpha}$. Let $H_{\alpha}$ be the generic filter for $\forceQ_{\alpha}$ over $W_1$. We also assume inductively that the $\alpha$-th $\omega_1$-block $\vec{T}^{\alpha}$ is still an independent sequence of Suslin trees in $W_1[H_{\alpha}]$, and that $W_1[H_{\alpha}] \models |P(\omega_1) \slash \NS)| = \aleph_2$. We want to describe the next forcing notion $\dot{\forceQ}_{\alpha}$ we will use.

\subsubsection{Towards a nice definition of stationarity}
We have to deal with two different cases, the first one concerns the definability of the nonstationary ideal. We assume that $\alpha$ is an even ordinal.
Let $(\alpha_1, \alpha_2) = F^{-1}(\alpha)$. We consider the universe $W_1[H_{\alpha_1}]$ and a fixed enumeration $(S^{\alpha_1}_i \, : \, i < \omega_2)$ of the stationary subsets of $\omega_1$ in $W_1[H_{\alpha}]$.  Then we pick the  $\alpha_2$-th element $S^{\alpha_1}_{\alpha_2}$ of $(S^{\alpha_1}_i \, : \, i < \omega_2)$ and code the characteristic function of  $S^{\alpha_1}_{\alpha_2}$ into the $\alpha$-th $\omega_1$ block $\vec{T}^{\alpha}= (T^{\alpha}_i \, : \, i < \omega_1)$ of our fixed independent sequence of Suslin trees $\vec{T}$. To be fully precise, we will not code the characteristic function of $S^{\alpha_1}_{\alpha_2}$ but instead the characteristic function of the following set 
$R^{\alpha_1}_{\alpha_2}:= \{ (\beta,0) \mid \beta \in S_{\alpha_2}^{\alpha_1} \}$, this is to distinguish the coding of stationary sets from coding reals according to our well-order of reals. 

We let $r^{\alpha_1}_{\alpha_2}: \omega_1 \rightarrow 2$ denote the characteristic function of $R^{\alpha_1}_{\alpha_2}$ . At stage $\alpha$ we will force with
$\dot{\forceQ}_{\alpha}^{G_{\alpha}} := \prod_{\beta < \omega_1} \forceP_{\beta}$ with finite support, where
\begin{equation*}
\forceP_{
\beta } = \begin{cases}
T^{\alpha}_{\beta}  & \text{ if } r^{\alpha_1}_{\alpha_2}(\beta)=1 \\ Sp(T^{\alpha}_{\beta}) & \text{ if } r^{\alpha_1}_{\alpha_2}(\beta)=0
\end{cases}
\end{equation*}
We let $h_{\alpha+1}$ be a generic filter for $\dot{\forceR}_{\alpha}$ over $W_1[H_{\alpha}]$, set $H_{\alpha+1}= H_{\alpha} \ast h_{\alpha+1}$ and continue. Note that $\dot{\forceQ}_{\alpha}$ has the countable chain condition as we assumed that $\vec{T}^{\alpha}$ is an independent sequence of Suslin trees in $W_1[H_{\alpha}]$ and by Lemma \ref{ccc Coding}. 

\subsubsection{Towards a $\Delta^1_4$-well-order of the reals}
In the second case, we assume that $\alpha$ is an odd ordinal and $F^{-1} (\alpha) = (\alpha_1,\alpha_2)$. We pass to $W_1[G_{\alpha_1}]$ and look at the $\alpha_2$-th pair of reals $(x_1,x_2) \in W_1 [G_{\alpha_1}]$. We additionally assume that the $<$-least name of $x_1$ (where $<$-denotes $M_1$'s global well-order) is below $x_2$-th least such name. We form the real 
$x^{\alpha_1}_{\alpha_2} = (1,x_1,x_2)$, and let $s^{\alpha_1}_{\alpha_2}$ be its characteristic functions and force with
\[ \dot{\forceQ}_{\alpha}^{G_{\alpha}} := \operatorname{Code} (s^{\alpha_1}_{\alpha_2}, \eta) \]
where $\eta$ is the least ordinal such that $\vec{T}^{\eta}$ is still an $\omega$-block of independent Suslin trees in our current ground model.

\subsection{Properties of the final model $W_2$}
We let $W_2:= W_1 [ G_{\omega_2} ]$, where $G_{\omega_2}$ is a $W_1$-generic filter for the just defined, $\omega_2$-length iteration $\forceP_{\omega_2}$ with finite support.
We shall prove now the property of $(\forceQ_{\alpha} \, : \, \alpha < \omega_2)$ which was assumed inductively in the definition of the iteration, namely that tails of $\vec{T}$ always remain an independent sequence of Suslin trees.

\begin{lemma}
Let $(\forceP_{\alpha}, \dot{\forceQ}_{\alpha} ) \, : \, \alpha \le \omega_2)$ be the just defined iteration over $W_1$ and let $\vec{T}$ be our sequence of independent Suslin trees. Let $\alpha< \omega_2$ and let $G_{\alpha}$ be a $\forceP_{\alpha}$-generic filter over $W_1$. We write $\vec{T}^{>\alpha'}$ for $(\vec{T}^{\beta} \,  :  \, \alpha' < \beta < \omega_2)$ where $\alpha'$ is the least ordinal such that the $\alpha'$-th $\omega$-block of $\vec{T}$ is still Suslin in $W_1[G_{\alpha}]$ . Then $W_1[G_{\alpha}]$ is a ccc extension of $W_1$ and $$ W_1[G_{\alpha}] \models \vec{T}^{>\alpha'} \text{ is an independent sequence of (blocks of) Suslin trees.}$$
Moreover at stage $\alpha$ every element of $\vec{T}^{\gamma}$, $\gamma < \alpha'$ has been destroyed.
\end{lemma}
\begin{proof}
This follows immediately from the definition of the forcing and the independence of the sequence $\vec{T}$.

\end{proof}
As a consequence the iteration $(\forceP_{\alpha} \, : \, \alpha \le \omega_2)$ is a finite support iteration of forcings with the countable chain condition, thus $\forceP_{\omega_2}$ has the countable chain condition as well.
This, plus the fact that $W_1\models 2^{\aleph_0} = 2^{\aleph_1} = \aleph_2$  readily gives that for every $\alpha < \omega_2$, if $H_{\alpha}$ denotes the generic filter for $\forceQ_{\alpha}$, then $W_1[H_{\alpha}] \models 2^{\aleph_1} = \aleph_2$. Indeed if $X \in W_1[H_{\alpha}] \cap P(\omega_1)$, then $X$ has a $\forceQ_{\alpha}$-name which is uniquely determined by a sequence $(A_{\eta} \, : \, \eta < \omega_1)$ such that every  $A_{\eta} \subset \forceQ_{\alpha}$ is a maximal antichain. There are $[\aleph_2]^{\aleph_0} = \aleph_2$-many antichains in $\forceQ_{\alpha}$, thus there are $\aleph_2$-many $\forceQ_{\alpha}$-names for subsets of $\omega_1$. As a consequence we will catch our tail after $\omega_2$ many stages.

\begin{lemma}
Let $G_{\omega_2}$ be an $W_1$-generic filter for $(\forceP_{\alpha}, \dot{\forceQ}_{\alpha} \, : \, \alpha < \omega_2)$. Then in $W_1[G_{\omega_2}]$, every stationary subset of $\omega_1$ is coded into an $\omega_1$-block of $\vec{T}$, i.e. for each $S$ stationary there is an $\alpha < \omega_2$ such that for every $\beta < \omega_1$ ( $\vec{T}^{\alpha}_{\beta}$ has a branch if and only if $\beta \in S$ and $\vec{T}^{\alpha}_{\beta}$ is special if and only if $\beta \notin S$). Here we write  $\vec{T}^{\alpha}_{\beta}$ for the $\beta$-th element of the $\alpha$-th $\omega_1$-block of Suslin trees $\vec{T}^{\alpha}$ from $\vec{T}$.
\end{lemma}
Recall that as $W_1[H_{\omega_2}]$ is a ccc extension of $W_0$, $H(\omega_2)^{W_0}$ is definable in $W_1[H_{\omega_2}]$ via the formula \begin{itemize}
\item[] $x \in H(\omega_2)^{W_0}$ if and only if $\exists M (M$ is presuitable and $x \in M)$.
\end{itemize}
As $\vec{T}$ is definable over $H(\omega_2)^{W_0}$, $\vec{T}$ is definable in $W_1[H_{\omega_2}]$ as well and so it can see the pattern that was written into $\vec{T}$. This hands us a new, definable predicate for stationary subsets of $\omega_1$. Counting quantifiers yields that the statement of the last Lemma is $\Sigma_2(\vec{C})$ over $H(\omega_2)$. Using the $\Sigma_1(\vec{C}, \vec{T}^0)$-definable set of suitable models from Lemma \ref{sigma_1 set of suitable models} will yield a $\Sigma_1(\vec{C}, \vec{T}^0)$-definition for stationarity in $W_1[H_{\omega_2}]$.
\begin{lemma}
There is a $\Sigma_1(\vec{C}, \vec{T}^0)$-formula $\Psi(S)$ which defines stationary subsets of $W_1[H_{\omega_2}]$:
\begin{itemize}
\item[$\Psi(S)$] if and only if there is an $\aleph_1$-sized, transitive model $N$ which contains $\vec{C}$ and $\vec{T}^0$ such that $N$ models that
\begin{itemize}
\item There exists a real $x$ such that $\Phi(x)$ holds, i.e. $x$ is a code for a suitable model $M$.
\item There exists an ordinal $\alpha$ in the suitable model $M$ such that $\vec{T}'$ is the $\alpha$-th $\omega_1$ block of the definable sequence of independent Suslin trees as computed in $M$ and $N$ sees a full pattern on $\vec{T}'$.
\item $\forall \beta < \omega_1$ $( \beta \in S$ if and only if $\vec{T}' (\beta)$ has a branch).
\item $\forall \beta < \omega_1$ $(\beta \notin S$ if and only if $\vec{T}' (\beta)$ is special).
\end{itemize}
\end{itemize}
Note that $\Psi(S)$ is of the form $\exists N (N \models...)$, thus $\Psi$ is a $\Sigma_1$-formula.
\end{lemma}

\begin{proof}
We assume first that $S \subset \omega_1$ is a stationary set in $W_1[H_{\omega_2}]$. Then by the last Lemma, $S$ is coded into an $\omega_1$-block of $\vec{T}$. The set of branches and specializing functions which witness that $S$ is coded into $\vec{T}$ is a set of size $\aleph_1$, thus there will be a suitable model $M$ which will contain it. We ensured that this suitable $M$ is coded by a real $x$ in $W_1[H_{\omega_2}]$. Then any uncountable, transitive $N$ which contains $r$ is a witness for $\Psi(S)$.

If, on the other hand $\Psi(S)$ is true then the suitable model $M$ which contains $S$ and witnesses that the characteristic function of $S$ is written into an $\omega_1$-block of its local $M$-Suslin trees is also a witness that the characteristic function of $S$ can be found in some block of $\vec{T}$ in the real world $W_1[H_{\omega_2}]$. This suffices as by definition of the iteration$ (\forceQ_{\alpha}\, : \, \alpha<  \omega_2)$, the patterns on $\vec{T}$ code up stationarity on $\omega_1$.
\end{proof}
This finishes the proof of the $\Delta_1 (\{\omega_1\} )$-definability of $\NS$ over $W_2$.

We finally turn to the $\Delta^1_4$-definable well-order of the reals of $W_2$.

\begin{lemma}
In $W_2$, there is a $\Delta_4^1$-definable well-order of its reals.
\end{lemma}

\begin{proof}
This follows from the fact that our formula asserting a real  $x$ ``is coded into $\vec{T}"$, denoted earlier by $\exists r \psi (x,r)$ is $\Sigma^1_4$. As every pair of reals gets listed by the bookkeeping in our iteration, and as we always either code $(x,y)$ or $(y,x)$ into $\vec{T}$, we know that if $x<y$ are two reals in $W_2$, then there will be a real $r$ such that $\psi((x,y),r)$ is true in $W_2$. If on the other hand $\exists r  (\psi ((x,y),r))$  uis true in $W_2$, then by lemma \ref{definable well-order of reals determines real world}, there is a pattern on $\vec{T}$ which is exactly the characteristic function of $(1,x,y)$, and as all the patterns on $\vec{T}$ we have in $W_2$ are created intentionally, we know that indeed $x<y$ must be true, so we found the $\Sigma^1_4$ and hence the $\Delta^1_4$-definition of the well-order.

\end{proof}

This finishes the proof of the  main theorem. We add as a remark that the coding procedure is actually more general and is not restricted to code $\NS$. Indeed, fixing $W_0$ as our ground model, we can pick any subfamily $\mathcal{F}$ of $P(\omega_1)$ and can code its members into the independent sequence $\vec{T}$. If $\mathcal{F}$ is definable and is preserved under ccc extensions (such as $\mathcal{F}= \NS^{+}$), then our proof works with almost no alterations. 
\section*{Acknowledgement}
This research was funded in whole by the Austrian Science Fund (FWF) Grant-DOI 10.55776/P37228.

\end{document}